\documentclass{amsart}
\usepackage{amssymb,amsmath,latexsym}
\usepackage{amsthm}
\usepackage{fontenc}
\usepackage{amssymb}
\numberwithin{equation}{section}

\newtheorem{theorem}{Theorem}[section]
\newtheorem{corollary}{Corollary}[theorem]

\begin{document}
\author{Alexander E. Patkowski}
\title{A remark on Fine's arithmetic functions}

\maketitle
\begin{abstract}In this note, we take another look at some arithmetic identities of N.J. Fine associated with divisor functions. We connect these functions with indefinite quadratic forms using a result due to Andrews. As a consequence, arithmetic theorems are extracted. \end{abstract}



\section{Introduction and main theorems} In Fine's monograph [3], several infinite products are paraphrased in terms of divisor functions. The two we are primarily concerned with herein are [3, pg.12,eq.(10.6)--(10.7)]
\begin{equation}\prod_{n=1}^{\infty} \frac{(1-q^{pn})^2}{(1-q^{pn-r})(1-q^{pn-p+r})}=\sum_{n=0}^{\infty}E_{r}(2pn+r(p-r);2p)q^n, \end{equation}
for natural numbers with $0<r<p,$ $(r,2p)=1,$ $|q|<1,$ and [3, pg.76, eq.(31.5)]
\begin{equation} q^r\left(\prod_{n=1}^{\infty} \frac{(1-q^{pn})^2}{(1-q^{pn-r})(1-q^{pn-p+r})}\right)^2=\frac{1}{p}\sum_{n=1}^{\infty}\left(\sum_{\substack{\delta d=pn-r^2 \\ d\equiv r\pmod{p}}}(d+\delta)\right)q^n,\end{equation}
for $0<r<p,$ $(r,p)=1,$ where \begin{equation} E_{r}(n;m)=\sum_{\substack{d|n \\ d\equiv r\pmod{m}}}1-\sum_{\substack{d|n \\ d\equiv -r\pmod{m}}}1.\end{equation}

The corollaries of (1.1)--(1.2) that follow in [3] are primarily connections between positive definite quadratic forms and divisor functions, which are classical and elegant in their own right. However, the product contained on the left sides of (1.1) and (1.2) is a rather special one, having an indefinite quadratic form expansion. This observation lead us to the following theorems, which we prove in the following section.
\begin{theorem}\label{thm:thm1} For $0<r<p,$ $(r,2p)=1,$
$$E_{r}(2pn+r(p-r);2p)=\sum_{\substack{k,l=-\infty \\k\ge|l|\\ n=p(k^2-l^2)/2+p(k+l)/2-lr}}^{\infty}(-1)^{k+l}.$$
\end{theorem}
Define 
$$D_{r,m}(n):=\sum_{\substack{d|n \\ d\equiv r\pmod{m}}}1,$$ so that by (1.3), $E_{r}(n;m)=D_{r,m}(n)-D_{-r,m}(n).$
We now state a corollary that follows from our first theorem.
\begin{corollary}\label{thm:Cor1} Let $0<r<p,$ $(r,2p)=1.$ We have that $D_{r,2p}(2pn+r(p-r))=D_{-r,2p}(2pn+r(p-r))$ if and only if $n$ is not of the form $Q_{r,p}(k,l):=p(k^2-l^2)/2+p(k+l)/2-lr,$ $k\ge|l|,$ $l,k\in\mathbb{Z},$ or the number of pairs $(k,l)$ of this form have an equal number which have $k+l\equiv0\pmod{2}$ as those which have $k+l\equiv1\pmod{2}.$ Moreover, $D_{r,2p}(2pn+r(p-r))$ is greater (resp. less) than $D_{-r,2p}(2pn+r(p-r))$ if the number of pairs $(k,l)$ which satisfy $n=Q_{r,p}(k,l),$ $k\ge|l|,$ $l,k\in\mathbb{Z},$ with $k+l\equiv0\pmod{2}$ are greater (resp. less) than those with $k+l\equiv1\pmod{2}.$
\end{corollary}
Our next theorem follows from (1.2).

\begin{theorem}\label{thm:thm2} For $0<r<p,$ $(r,p)=1,$
$$\frac{1}{p}\sum_{\substack{\delta d=pn-r^2 \\ d\equiv r\pmod{p}}}(d+\delta)=\sum_{\substack{k_1,l_1,k_2,l_2=-\infty \\k_1\ge|l_1|, k_2\ge|l_2|\\ n=p(k_1^2-l_1^2)/2+p(k_1+l_1)/2-l_1r+p(k_2^2-l_2^2)/2+p(k_2+l_2)/2-l_2r+r}}^{\infty}(-1)^{k_1+l_1+k_2+l_2}.$$
\end{theorem}
Next we have the equivalent of a non-negativity result concerning certain indefinite quaternary quadratic forms.

\begin{corollary}\label{thm:Cor2} Let $0<r<p,$ $(r,p)=1.$ We have that, for each natural number $n,$ the number of quadruples $(k_1,k_2,l_1,l_2)$ which satisfy $n=Q_{r,p}(k_1,l_1)+Q_{r,p}(k_2,l_2)+r,$ $k_1\ge|l_1|,$ $k_2\ge|l_2|,$ $l_1,l_2,k_1, k_2\in\mathbb{Z},$ which have $k_1+k_2+l_1+l_2\equiv0\pmod{2}$ are greater than or equal to those which have $k_1+k_2+l_1+l_2\equiv1\pmod{2}.$
\end{corollary}

\section{proofs of results}
For $1<|z|<|q|^{-1},$ $|q|<1,$ Andrews [1, Lemma 1] discovered the expansion,
\begin{equation} \prod_{n=1}^{\infty} \frac{(1-q^{n})^2}{(1-zq^{n})(1-z^{-1}q^{n-1})}=\sum_{\substack{k,l=-\infty \\k\ge|l|}}^{\infty}(-1)^{k+l}z^{l}q^{(k^2-l^2)/2+(k+l)/2}.\end{equation}

\begin{proof}[Proof of Theorem~\ref{thm:thm1}] This formula follows from replacing $q$ by $q^{p}$ in (2.1) and then $z=q^{-r},$ where $0<r<p,$ $(r,2p)=1.$ Equating coefficients of $q^{n}$ with (1.1) now gives the result.

\end{proof}

\begin{proof}[Proof of Corollary~\ref{thm:Cor1}] This follows from inspecting the weight in the sum on right hand side of Theorem 1.1, and noting that $D_{\pm r, p}(n)$ is a non-negative quantity.

\end{proof}

\begin{proof}[Proof of Theorem~\ref{thm:thm2}] We use the $(z,q)\rightarrow(q^{-r},q^{p})$ case of (2.1) we obtained in the proof of Theorem 1.1 in (1.2), and then equate coefficients of $q^n$ with (1.2).

\end{proof}

\begin{proof}[Proof of Corollary~\ref{thm:Cor2}] It is enough to show the left side of Theorem 1.2 is non-negative. One way to see this is to use the re-writing found in [2, Proposition 1, eq.(8)],
$$\frac{1}{p}\sum_{\substack{ d|np-r^2, d>0 \\ d\equiv r\pmod{p}}}(d+\frac{np-r^2}{d}).$$
Notice that if the summand were positive then so would $d^2+np-r^2.$ Since $d\equiv r\pmod{p},$ there exists integers $m\ge0,$ such that $d^2+np-r^2=(pm+r)^2+np-r^2=(pm)^2+2pmr+np.$ Since $np>0,$ it follows that the sum is non-negative.

\end{proof}

\section{Observations} 
It is interesting to mention that our results coupled with [2] suggest that indefinite quadratic forms seem to have some natural connection with Klein forms. This seems to follow naturally from the fact that Klein forms have integral weight, and Hecke modular forms associated with indefinite quadratic forms have weight one [4]. The example presented here of this phenomenon being (1.1), which is a holomorphic modular form of weight one [2, Theorem 2, eq.(1)]. The product on the left side of (1.2), the square of (1.1), has also been noted as a holomorphic modular form of weight two [2, Theorem 2, eq.(8)].

1390 Bumps River Rd. \\*
Centerville, MA
02632 \\*
USA \\*
E-mail: alexpatk@hotmail.com, alexepatkowski@gmail.com
\end{document}